\documentclass[12pt]{article}
\usepackage[utf8]{inputenc}
\usepackage[english]{babel}
\usepackage{amsmath}
\usepackage{amsthm}
\usepackage[usenames]{color}
\usepackage{amsfonts}
\usepackage{amssymb}
\usepackage{makeidx}
\usepackage{graphicx}
\usepackage[left=2.6cm,right=2.6cm,top=3.5cm,bottom=2cm]{geometry}
\usepackage{enumerate}
\usepackage{mathrsfs}
\usepackage{multicol}

\usepackage{todonotes}

\usepackage{hyperref}

\newtheorem{theorem}{Theorem}[section]

\newtheorem{prop}[theorem]{Proposition}
\newtheorem{lemma}[theorem]{Lemma}
\newtheorem{corollary}[theorem]{Corollary}
\newtheorem{qs}[theorem]{Question}

\theoremstyle{definition}
\newtheorem{definition}{Definition}[section]
\newtheorem{remark}[theorem]{Remark}

\newcommand{\Z}{\mathbb{Z}}
\newcommand{\R}{\mathbb{R}}

\newcommand{\dist}{\mathrm{dist}}

\newcommand{\tF}{\widetilde{F}}

\newcommand{\notto}{\,\,\, {\not\!\!\longrightarrow}\,}

\title{An exotic Denjoy example of class~$C^1$}

\author{\textsc{Maximiliano Escayola}\thanks{School of Mathematics, Korea Institute for Advanced Study (KIAS), Seoul, 02455, Korea.
\\ 
The author was supported by KIAS individual grant no.~MG107601 of Korea Institute for Advanced Study.
\\ e-mail: maxiescayola@kias.re.kr}}
\date{}

\begin{document}

\maketitle

\begin{abstract}
We construct a $C^1$-diffeomorphism of the circle having an invariant Cantor minimal set such that, when the lengths of the connected components of its complement are arranged in decreasing order, the ratios of consecutive lengths do not converge~to~1. This gives a negative answer to a longstanding question posed by Dusa McDuff.
\end{abstract}

\vspace{0.2cm}

\noindent{\bf Keywords:}  Exceptional circle diffeomorphism, Denjoy counterexample, McDuff's question.

\vspace{0.2cm}

\noindent{\bf 2020 Mathematics Subject Classification:}  37E10 (Primary); 37C05 (Secondary)

\section{Introduction}
\subsection{McDuff's question and the main result}
Given an orientation-preserving homeomorphism $f$ of the circle $\mathbb S^1=\R/\Z$, its Poincaré rotation number 
\[
\lim_{n\to\infty}\rho(f)=\frac{F^n(x)-x}{n} \mod\Z,
\] where $x\in \R$ is arbitrary and $F$ is any lift of $f$, provides a fundamental description of its dynamics. If $\rho(f)$ is rational, then $f$ has periodic points. If $\rho(f)$ is irrational, then $f$ is semiconjugate to the corresponding irrational rotation, and there are two possibilities: either the action of $f$ is minimal, meaning that every orbit is dense, in which case the semiconjugacy is actually a conjugacy, or $f$ admits a unique minimal invariant Cantor set on which its action is minimal. Such an $f$ is usually called an exceptional homeomorphism, or an exceptional diffeomorphism when $f$ is differentiable. 

The classical Denjoy theorem \cite{denjoy} states that the latter possibility cannot occur in the~$C^2$ setting. In other words, every $C^2$-diffeomorphism of the circle with irrational rotation number is conjugate to the corresponding irrational rotation. Examples of exceptional $C^1$-diffeomorphisms had already appeared in the work of Bohl \cite{bohl}. Later, Herman~\cite{herman} showed that exceptional $C^{1+\alpha}$-diffeomorphisms exist for every $\alpha\in (0,1)$. Here, by a $C^{1+\alpha}$-diffeomorphism, we mean a $C^1$-diffeomorphism whose derivative is $\alpha$-Hölder continuous. More recently, Kim and Koberda \cite{kim-koberda} substantially generalized these constructions by establishing an integrability condition on the modulus of continuity of the derivative that guarantees the existence of exceptional circle diffeomorphisms.

For every positive summable sequence $(\ell_n)_{n\in\Z}$ satisfying $\ell_n/\ell_{n+1}\to 1$, $n\to \pm \infty$, Dusa McDuff~\cite{McDuff} constructed an exceptional $C^1$-diffeomorphism $f$ of the circle such that the connected components of the complement of its invariant Cantor set are the intervals $I_n=f^n(I_0)$, with $|I_n|=\ell_n$. In her examples, if the lengths $\ell_n$ are rearranged in decreasing order, say  $\lambda_{1}\ge \lambda_{2}\ge\dots,$ then $\lambda_n/\lambda_{n+1}\to 1.$ More generally, McDuff investigated necessary conditions for a Cantor subset of the circle to arise as the minimal set of a $C^1$-diffeomorphism (the interested reader may consult the works of Portela \cite{portela-new,portela-regular} for further developments in this direction). She proved that the sequence $(\lambda_n/\lambda_{n+1})_{n\geq 1}$ is bounded and has $1$ as a nontrivial limit point. In particular, 
\[
\liminf_{n\to\infty}\frac{\lambda_n}{\lambda_{n+1}}=1.
\]
This led her to ask whether the full sequence must always converge to $1$:
\begin{qs}[D. McDuff,~{\cite{McDuff}}] Let $f$ be an exceptional $C^1$ circle diffeomorphism, and let $\lambda_1\geq\lambda_2\geq\cdots$ be the lengths of the connected components of the complement of its minimal invariant Cantor set, arranged in decreasing order. 
\[ 
\text{Is it always true that}\quad\lim_{n\to\infty}\frac{\lambda_n}{\lambda_{n+1}}=1\quad? 
\] 
\end{qs} 
\noindent Results pointing in the affirmative direction of the answer to the McDuff's question were obtained by Iglesias and Portela \cite{iglesias-portela-nonlinearity,iglesias-portela} for certain families of Cantor sets. 
We refer the reader to the expository paper of Athanassopoulos \cite{A}, which is devoted to this question, for further background.

The main result of this work is an example, providing  a \emph{negative} (and thus unexpected) answer to McDuff's question. Namely, we establish the following theorem.

\begin{theorem}[Main result]\label{thm: counterexample McDuff}
There exists an orientation-preserving $C^1$ diffeomorphism
\[
f:\mathbb S^1\to\mathbb S^1
\]
possessing a minimal invariant Cantor set $K$ with the following property: if
$\lambda_1\geq\lambda_2\geq\cdots$
are the lengths of the connected components of $\mathbb S^1\setminus K$,
arranged in decreasing order, then
\[
\limsup_{i\to\infty}\frac{\lambda_i}{\lambda_{i+1}}>1;
\quad \text{in particular},\quad\frac{\lambda_i}{\lambda_{i+1}}
\notto 1.
\]
\end{theorem}

\subsection{Example: statement of the construction}

Theorem~\ref{thm: counterexample McDuff} is proved by an explicit construction, modifying a construction by V. Kleptsyn and A. Navas, proposed in~\cite{KN}. This construction provides Denjoy examples of class $C^1$, that are affine on the intervals the complement to the Cantor minimal set (that is of Lebesgue measure zero in their construction).

Namely, assume that we are given an irrational number $\alpha$ and a summable sequence of lengths~$(\ell_n)_{n\in \Z}$:
\begin{equation}\label{eq: sum lengths}
S:=\sum_{n\in\Z}\ell_n <\infty.
\end{equation}
Consider the orbit of the point $0$ under the rotation
\[R_\alpha:\mathbb S^1\to \mathbb S^1\quad\text{defined by}\quad R_\alpha(x):=x+\alpha \mod \Z,\]
that is, 
\[
x_n=(n\alpha \mod \Z)\in \mathbb S^1=\R/\Z.
\]

\begin{figure}
    \centering
    \includegraphics[width=0.5\linewidth]{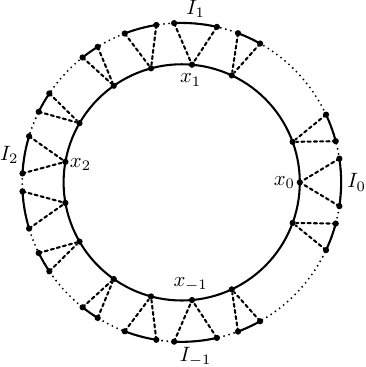}
    \caption{Blow-up of the orbit of the rotation by~$\alpha$}
    \label{fig:Denjoy}
\end{figure}

Then, one can consider a circle of length~$S$, on which a set of full measure is given by disjoint intervals $I_n$ of lengths $\ell_n$, whose circular order coincides with the that of the points~$x_n$ (see Figure~\ref{fig:Denjoy}). In other words, one considers the blow-up of the orbit $(x_n)$, replacing each point of this orbit with an interval of the length $\ell_n$. Then, one performs a continuous change of variables, given by the distribution function of the measure that is the sum of the Lebesgue measures inside these intervals, to ensure that the Cantor set that is the to their union complement is of measure zero.

On this circle $\R/S\Z$, one can consider the homeomorphism $f$, defined by the property that $f(I_n)=I_{n+1}$ for all~$n$, and that the restrictions $f|_{I_n}:I_n\to I_{n+1}$ are exactly affine. This map is semi-conjugate to the rotation by $\alpha$ via a map $p$ that collapses each $I_n$ to the corresponding point $x_n$ (and extended to the full circle by continuity): 
\[
p\circ f = R_{\alpha}\circ p.
\]

In their construction in~\cite{KN}, that we briefly recall in Section~\ref{sec: initial construction} below, V.~Kleptsyn and A.~Navas describe a way to choose lengths $\ell_n$ (for any given irrational $\alpha$) so that the resulting map $f$ is a $C^1$-diffeomorphism. The construction of an example that proves Theorem~\ref{thm: counterexample McDuff} is a modification of theirs one, that we explain in Section~\ref{sec: modification} below. 

Once implemented, this modification boils down to the following. We fix a particular value of the angle, 
\[
\alpha=\sqrt{5}-2=[0;\overline{4}], 
\]
and denote $r=\alpha^{1/2}$. Let~$p_m/q_m$ be the sequence of good approximations to~$\alpha$; the denominators $q_m$ then are recurrently defined by
\[
q_{m+1}=4q_m + q_{m-1}, \quad q_0=1, \quad q_{-1}=0.
\]

Now, consider the ``chainsaw'' function $d:\R\to\R_+$, defined by 
\[
d(x):=\dist(x,\Z) = \min (\{x\},\{1-x\}),
\]
and let the sequence $(L_n)_{n\in \Z}$ be defined (for the reasons that will be explained below) as the sum of the series
\begin{equation}\label{eq: chosen L}
L_n:=\sum_{j=1}^{\infty} r^{-j}\cdot d\left(\alpha^{j} n\right).
\end{equation}
We now define the sequence of lengths $(\ell_n)_{n\in\Z}$ by
\begin{equation}\label{eq: lengths}
\ell_n:=e^{-L_n}.    
\end{equation}
Note that $L_{-n}=L_n$, as $d(x)$ is an even function. The well-definedness of the sequence of lengths is guaranteed by the following statement (that will be proven below).
\begin{lemma}\label{l: convergence and summability}
    For every $n\in \Z$, the series~\eqref{eq: chosen L} converges. Also, the sequence $(\ell_n)$, defined by~\eqref{eq: lengths}, is summable.
\end{lemma}

The choices of the angle $\alpha$ and of the lengths $\ell_n$ fix the corresponding homeomorphism~$f$. Theorem~\ref{thm: counterexample McDuff} will be proven, once we show that this map satisfies its conclusions.

\begin{proof}[Scheme of the proof of Theorem~\ref{thm: counterexample McDuff}]
Our first statement is that the quotients of consecutive lengths $\ell_n/\ell_{n+1}$ do not tend to~$1$. Namely, for every $k\geq0$, set 
\begin{equation}\label{eq: define Mk}
M_k=q_0+q_1+q_2+\dots+q_k.
\end{equation} 
We then have the following statement, that is 
the main tool in the proof of Theorem~\ref{thm: counterexample McDuff}.

\begin{prop}\label{prop: arithmetic gap}
For every $k\geq0$, one has
\[
\max_{|n|\leq M_k}L_n=L_{M_k}
\quad
\text{and}
\quad
\min_{|n|\geq M_k+1}L_n=L_{M_k+1}.
\]
Moreover,
\[
L_{M_k+1}-L_{M_k}=L_1.
\]
\end{prop}

This proposition implies the upper limit part in Theorem~\ref{thm: counterexample McDuff}. Indeed, it implies that for every index $i$ of the form $i_k=2M_k+1$, one has 
\begin{equation}
\lambda_{i_k}=e^{-L_{M_k}}, \quad \lambda_{i_k+1}=e^{-L_{M_k+1}}:
\end{equation}
the lengths $\lambda_i$ not shorter than $e^{-L_{M_k}}$ correspond to $n=-M_k,\dots,M_k$, and the shorter ones start with~$\ell_{M_{k+1}}$. Thus,
\[
\frac{\lambda_{i_k}}{\lambda_{i_k+1}} = e^{L_1}.
\]
In particular, 
\[
\limsup_{i\to \infty}\frac{\lambda_{i}}{\lambda_{i+1}} \ge e^{L_1};
\]
moreover, as we will see from the proof, the inequality here actually is an equality. We illustrate Proposition~\ref{prop: arithmetic gap} in Figure~\ref{fig:graph}, where log-lengths $\log \ell_n$ are plotted (for $n\ge 0$; the other half is symmetric, $\ell_{-n}=\ell_n$). The reader can see the identical gaps, implied by Proposition~\ref{prop: arithmetic gap}, that correspond to the first few indices~$n=M_k$ (namely, to $n=0,1,5,22,94$). 

\begin{figure}[h!]
    \centering
    \includegraphics[width=0.9\linewidth]{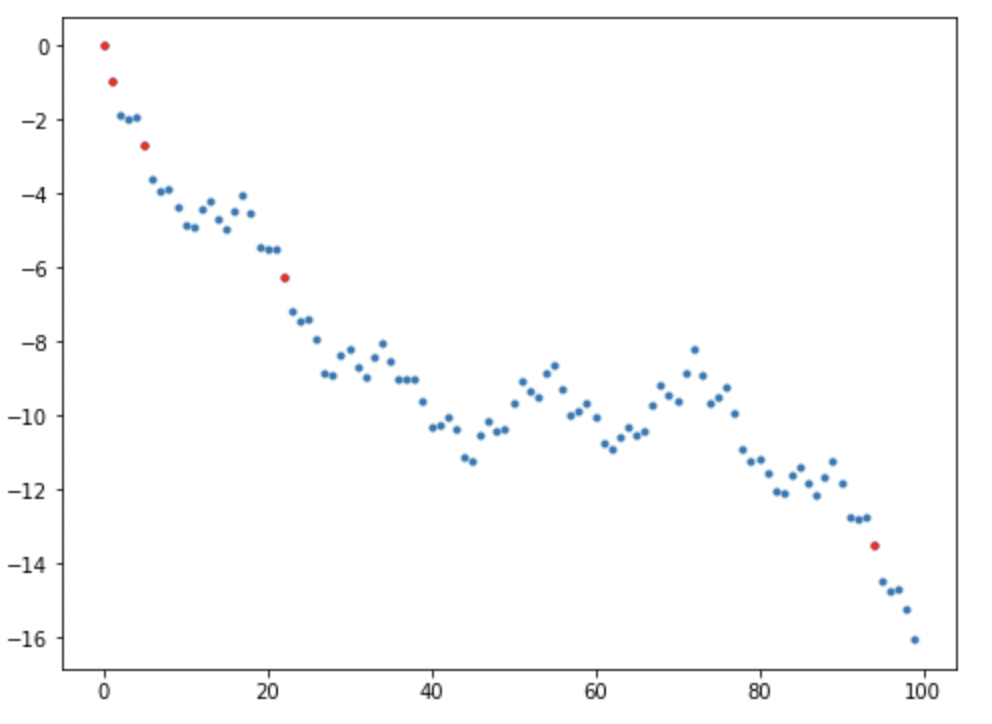}
    \caption{Log-lengths $\log \ell_n=-L_n$ for $n=0,1,\dots,99$; points, corresponding to the subsequence $n=M_k$, are marked by red.}
    \label{fig:graph}
\end{figure}

The second part of the proof of Theorem~\ref{thm: counterexample McDuff} is the $C^1$-regularity of the constructed map~$f$. In the same way as in~\cite{KN}, note that the map~$f$ is affine on each interval~$I_j$, with a constant derivative inside it:
\[
\log Df|_{\mathrm{int} \, I_j}= \log \frac{\ell_{j+1}}{\ell_j}.
\]
As the union of interiors of intervals $I_j$ is a set of full measure on the circle, the map $f$ is a $C^1$-diffeomorphism if and only if these values of $\log Df$ extend to a continuous function of the circle. As $\log Df$ is constant inside each interval $I_j$, such a continuous function should be of the type $\varphi \circ p(y)$, where $p:\R/S\Z \to \R/\Z$ is the projection that collapses each interval $I_j$ into the point~$x_j$ and semi-conjugates~$f$ to the rotation $R_{\alpha}$, and $\varphi$ is a continuous function on the circle~$\R/\Z$. 

Vice versa, if the function on the orbit $(x_n)$, defined by 
\begin{equation}\label{eq: log Df}
\varphi(x_n)= \log \ell_{n+1} - \log \ell_n,    
\end{equation}
extends to a continuous function on $\R/\Z$, then $f$ indeed is a $C^1$-diffeomorphism with $\log Df|_y=\varphi\circ p(y)$. The following statement confirms that this is indeed the case.
\begin{prop}\label{prop: continuous logarithmic slope - new}
    The function $\varphi$, defined on the orbit $(x_n)$ of rotation $R_{\alpha}$ by~\eqref{eq: log Df}, extends to a continuous function on the circle.
\end{prop}
We illustrate this proposition by a numerical simulation of the function $\varphi$: its graph is shown on Figure~\ref{fig:graph}.

\begin{figure}[h!]
    \centering
    \includegraphics[width=0.9\linewidth]{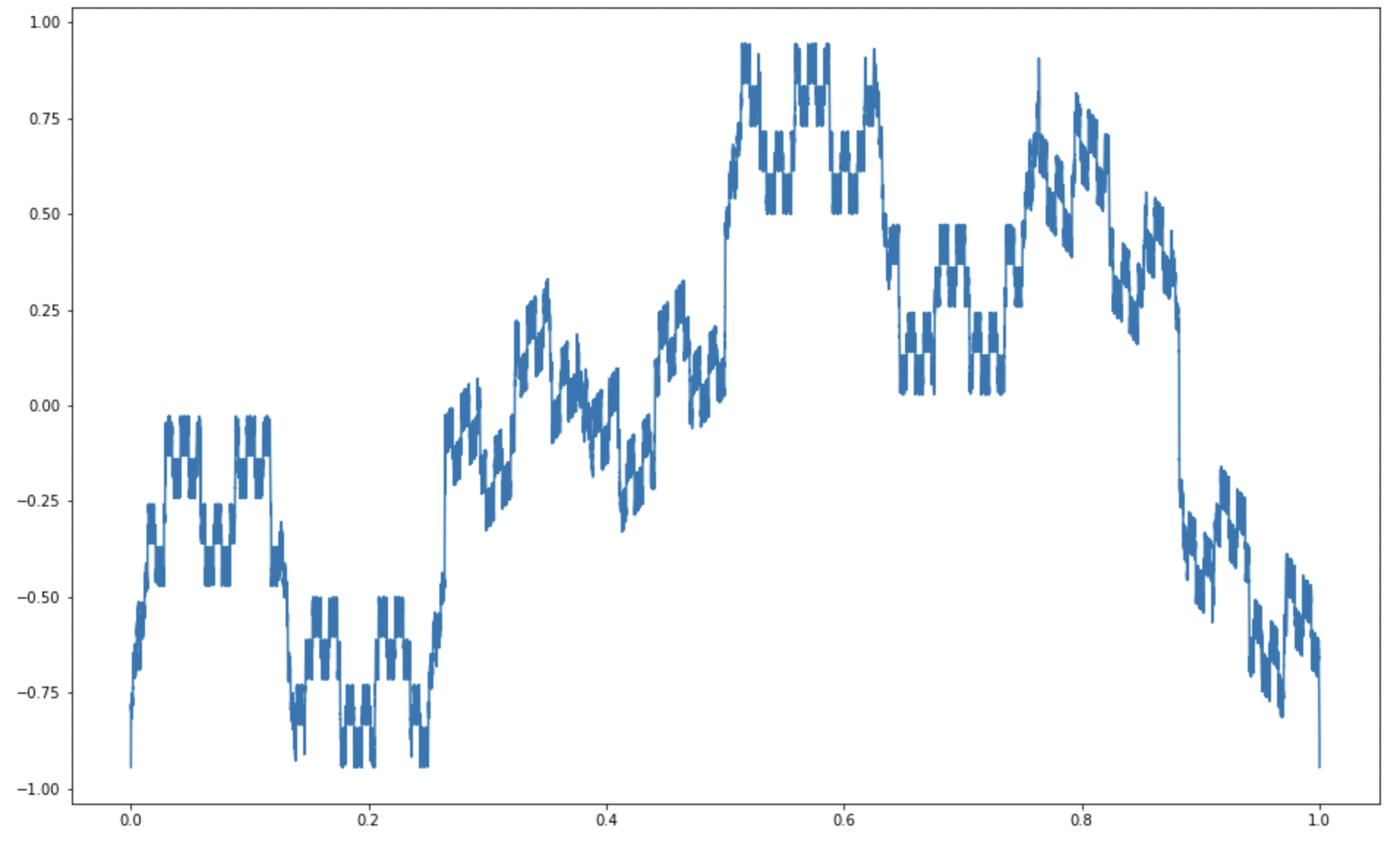}
    \caption{The graph of $\log Df|_y$ as a function of $p(y)$ (equivalently, in the coordinates after the semi-conjugacy to the rotation by~$\alpha$).}
    \label{fig:log-DF}
\end{figure}

We provide some necessary preliminaries for good approximations for~$\alpha$, recall Kleptsyn and Navas' construction and explain our ideas of its modification in Section~\ref{s: circle}, proving Proposition~\ref{prop: continuous logarithmic slope - new} (conditionally to Lemma~\ref{l: convergence and summability}) at the end of Section~\ref{sec: modification}. Next, we pass to the study of lengths $\ell_n$ in Section~\ref{sec: sequence}, proving Lemma~\ref{l: convergence and summability} in Section~\ref{sec: Definition and summability of the lengths} and then Proposition~\ref{prop: arithmetic gap} in Section~\ref{sec: gaps}. Once these three statements are established, the proof of Theorem~\ref{thm: counterexample McDuff} is complete.
\end{proof}

\section{Ideas of the construction}\label{s: circle}

\subsection{Preliminaries on good approximations}

For a given $\alpha=[a_0;a_1,a_2\ldots]\in \R\smallsetminus\mathbb Q$, it is known that its sequence of truncated continued fractions $p_j/q_j=[a_0;a_1,\ldots,a_j]$ satisfies the recurrence relations
\begin{equation}\label{eq: p_j,q_j}
p_{j+1}=a_{j+1}p_j+p_{j-1}\quad\text{and}\quad q_{j+1}=a_{j+1}q_j+q_{j-1},
\end{equation}
with the conventions $p_{-1}=1$, $p_0=a_0$, $q_{-1}=0$ and $q_0=1$. 

Recall that we have fixed $\alpha=\sqrt{5}-2=[0;\overline{4}]$; for this continued fraction,~\eqref{eq: p_j,q_j} yields
\begin{equation}\label{eq: 4q_j+q_{j+1}}
q_{-1}=0,\qquad q_0=1,\quad\text{and}\quad q_{j+1}=4q_j+q_{j-1}\quad\text{for every}\quad j\geq0
\end{equation}
and $p_j=q_{j-1}$ for every $j\ge 0$.

In the following lemma, we state the properties of the sequence of denominators $(q_j)_{j\geq -1}$ that will be used throughout the construction. 

\begin{lemma}\label{lem: exact approximation}
For every $j\geq0$, one has 
\[
q_j\alpha-q_{j-1}=(-1)^j\alpha^{j+1}.
\]
In particular, $\dist(q_j\alpha,\Z)=\alpha^{j+1}.$
\end{lemma}

\begin{proof}
The characteristic equation $\lambda^2-4\lambda-1=0$ for the recurrence relation~\eqref{eq: 4q_j+q_{j+1}} has two roots, $\alpha^{-1}$ and~$(-\alpha)$, and hence the sequence $(q_j)$ is a linear combination of geometric progressions with these ratios. From $q_{-1}=0$, one gets
\[
q_j=c_0 \cdot (\alpha^{-(j+1)} - (-\alpha)^{j+1}), 
\]
for some $c_0$ (actually, $q_0=1$ implies $c_0=1/(\alpha + \alpha^{-1})$). In the linear combination $q_j \alpha - q_{j-1}$, the geometric progression with the ratio $\alpha^{-1}$ gets canceled, leading to 
\[
q_j \alpha - q_{j-1} = -c_0 (\alpha^{-1}+\alpha) \cdot (-\alpha)^{j+1},
\]
thus implying the conclusion of the lemma (the constant factor can be checked either directly, or from evaluating both sides for $j=0$).

\end{proof}

\subsection{Initial construction}\label{sec: initial construction}

We briefly recall the construction of Kleptsyn and Navas~\cite{KN}, in which they extend the functions~$\varphi$ in~\eqref{eq: log Df} continuously to the whole circle while controlling the summability of the sequence of lengths $\ell_n$. Their construction works for every value of $\alpha$; we need here only its simplified version for $\alpha$ of bounded type. 

Namely, let $\alpha\in\mathbb R\setminus\mathbb Q$, and choose integers $Q_m,P_m$ such that $\beta_m:=Q_m\alpha-P_m$
is small. The construction of~\cite{KN} starts with considering a \emph{formal} Fourier-type series 
\begin{equation}\label{eq: psi-def}
    \psi(x) := \sum_{m\geq 1} d_m \big( \cos (2\pi Q_m x) -1 \big).
\end{equation}
for some (reasonably large) choice of the coefficients~$d_m$.
This series consists of nonpositive terms and (due to $d_m$ being large) diverges almost everywhere. However, it converges at $x=0$. One also can consider the formal (term-by-term) increment of the right hand side of~\eqref{eq: psi-def} over a rotation by $\alpha$,
\begin{equation}\label{eq: phi-def}
    \varphi(x):=\sum_{m\geq 1} d_m \big( \cos (2\pi Q_m (x+\alpha)) -\cos (2\pi Q_m x) \big).
\end{equation}
Taking such an increment (roughly) multiplies the coefficient for the Fourier frequency $Q_m$ by~$\beta_m$. Hence, if the products $b_m:=d_m\beta_m$ decrease sufficiently fast so that the series $\sum_m b_m$ would be absolutely convergent, the Fourier series~\eqref{eq: phi-def} converges uniformly. In this case, its sum $\varphi(x)$ is a continuous function on the circle. Hence, in that case the sum $\psi(x)$ is finite at any point of the orbit $(\{n\alpha\})_{n\in \Z}$ of $0$ under the rotation by~$\alpha$. The authors of~\cite{KN} then choose the log-lengths
\begin{equation}\label{eq: log L}
\log(\ell_n):=\psi(n\alpha)=\sum_m d_m\left(\cos(2\pi Q_m n\alpha)-1\right),    
\end{equation}
leading to the diffeomorphism $f$ with $\log Df = \varphi \circ p$, where $p$ is the semiconjugacy to the rotation. 

The construction is completed by fixing the choice of the frequencies $Q_m$ and the coefficients $d_m$ in such a way that the series $\sum_n \ell_n$ converges. The idea here is that for every $n$, there will be one of the summands in the right hand side of~\eqref{eq: log L} that will be sufficiently large. 

For a particular case of $\alpha$ of bounded type, it suffices to take as frequencies $Q_m=q_m$, the denominators of good approximations. Indeed, in this case $\beta_m$ is comparable to $1/q_m$, and the set of possible indices $n$ is split into segments of values that are comparable to $q_m$, for which 
\[
1-\cos 2\pi \beta_m n\ge c
\]
for some constant $c>0$. The latter implies that $\log \ell_n \le - c d_m$, and the convergence follows (see~\cite[Section~2.5]{KN} for the case where $\alpha$ is the golden mean).

\subsection{Modification}\label{sec: modification}
We start by noticing that the essential mechanism ensuring the convergence of $\varphi$ in the preceding construction does not depend on the particular choice of the cosine function; it only requires a (nonpositive) $1$-periodic Lipschitz function $g(x)$, vanishing at~$0$. Namely,~\eqref{eq: psi-def} can be written in terms of the function $g(x)=\cos 2\pi x -1$ as 
\begin{equation}\label{eq: psi from g}
\psi(x) = \sum_m d_m g(Q_m x),    
\end{equation}
and thus~\eqref{eq: phi-def} as 
\begin{equation}\label{eq: phi from g}
\varphi(x) = \sum_m d_m (g(Q_m x+\beta_m) -g(Q_m x)).    
\end{equation}
If $g$ is a Lipschitz function, the summands in the right hand side of~\eqref{eq: phi from g} are bounded uniformly by
\[
\left\|d_m\left(g(Q_mx+\beta_m)-g(Q_mx)\right)\right\|_\infty\leq \text{const}\cdot d_m|\beta_m|.
\]
Hence, if $d_m=b_m/|\beta_m|$ for some positive summable sequence $(b_m)$, the series~\eqref{eq: phi from g} converges uniformly on the circle,
\begin{equation}\label{eq: derivative norm}
\sum_m\left\|d_m\left(g(Q_mx+\beta_m)-g(Q_mx)\right)
\right\|_\infty\leq \text{const}\sum_m b_m<\infty,    
\end{equation}
and thus defines a continuous function~$\varphi$. Also, as $g(0)=0$, the series~\eqref{eq: psi from g} converges at $x=0$, and thus (due to the finiteness of the increments $\varphi(x)$) at every point of $x=n\alpha$ of its orbit. Thus, the lengths $\ell_n$, given by 
\[
\log \ell_n := \psi(n\alpha) = \sum d_m g(Q_m \cdot n\alpha),
\]
are well-defined. Finally, the summability condition for the lengths~$\ell_n$ is still to be verified separately.

In our construction, we replace the function $g(x)=\cos 2\pi x - 1$ by the function 
\[
g(x):=-d(x),\qquad d(x):=\dist(x,\mathbb Z).
\]
One of the main advantages of using the function $d(\cdot)=\dist(\cdot,\Z)$ in the construction is its piecewise-affine structure. On certain carefully chosen intervals, this allows us to obtain explicit affine expressions for the logarithmic lengths and thereby establish the required gap properties. See section \ref{sec: gaps} below.

Thus, we define, for all $n\in\mathbb Z$
\begin{equation}\label{eq: 1st lenghts}
\log(\ell_n)=\sum_m d_m\cdot g(n\beta_m)=-\sum_m d_m\cdot d(n\beta_m).
\end{equation}

We also fix the choice $d_m=r^{-m}$, and $Q_m=q_m$, $P_m=p_m=q_{m-1}$. As for our choice of $\alpha=[0;\overline{4}]$ one has $\beta_m=(-1)^{m-1} \alpha^m$ (see Lemma~\ref{lem: exact approximation}), the expression~\eqref{eq: 1st lenghts} with these choices becomes exactly the choice of lengths given by~\eqref{eq: lengths} and \eqref{eq: chosen L} in the introduction.
The above arguments immediately imply that, conditionally to Lemma~\ref{l: convergence and summability} (that is needed to ensure that the example actually is well-defined), the resulting map $f$ is indeed $C^1$-regular.

\begin{proof}[Proof of Proposition~\ref{prop: continuous logarithmic slope - new}]
    One has $b_m=|\beta_m|\cdot d_m =\alpha^m r^{-m} = r^m$. Thus, the series $\sum_m b_m$ converges (as a geometric series). Hence, due to~\eqref{eq: derivative norm}, so does the series~\eqref{eq: phi from g} for the function~$\varphi$. This implies that the log-derivatives $\log Df$ Lebesgue-almost everywhere (that is, inside the union of the intervals $I_j$) coincide with the continuous function $\varphi\circ p$. Together with the continuity of~$f$, this implies that $f$ is a $C^1$-diffeomorphism with $\log Df = \varphi \circ p$.
\end{proof}

\section{Sequence of lengths}\label{sec: sequence}

In order to study the behaviour of the sequence of log-lengths $\log \ell_n=-L_n$, we change the viewpoint on the series~\eqref{eq: 1st lenghts}, defining it. 

Previously, in order to ensure that the function $f$ is of class $C^1$, as in following~\cite{KN}, it was considered as a sequence of values of the function $\psi$ on the circle on the orbit of rotation by $\alpha$. This function $\psi$ was a sum of a series~\eqref{eq: psi from g} with more and more oscillating general terms $d_m g(Q_m x)$, and with the coefficients $d_m$ not tending to zero, this series diverges almost everywhere, making its values over the orbit $(n\alpha)_{n\in \Z}$, where the values are finite, quite hard to study. 

Now, the new viewpoint that we introduce is that the series~\eqref{eq: 1st lenghts} can be seen as a value of a function on the real line, given by the sum of the series
\begin{equation}\label{eq: F series}
\sum_m d_m g(\beta_m x),    
\end{equation}
evaluated at an integer point $x=n$. Contrary to the function $\psi$, the series~\eqref{eq: F series} can, due to $g(0)=0$ and due to the rescaling by an exponentially small factor $\beta_m$, converge everywhere on $\R$. This will actually be the case for our choice of functions $g(x)=-d(x)$, $d_m=r^{-m}$ and of the angle $\alpha$: we will see that the resulting sum~\eqref{eq: F series} has nice properties, allowing to analyze it. 

On Figure~\ref{fig: d and its rescaling}, one can see the graphs of the function $d(x)$ and the first two its rescaled copies $r^{-m}d(\alpha^m x)$. It is easy to notice that the sum~\eqref{eq: F series} is thus a piecewise-linear function on every finite interval: for every interval $[0,A]$, there are only a finite number of summands that are not exactly linear on it.

We thus define the function
\begin{equation}\label{eq: define F}
F:\R\to \R, \quad F(x)=\sum_{j=1}^{\infty} r^{-j} d(\alpha^jx),    
\end{equation}
so that $\log \ell_n = - L_n = -F(n)$.

\begin{figure}
    \centering
    \includegraphics[width=0.9\linewidth]{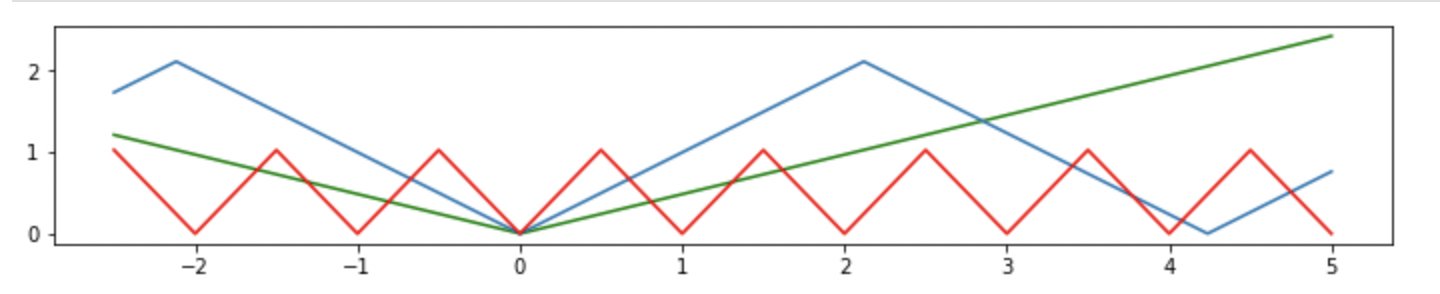}
    \caption{The function $d(x)$ and its rescalings $r^{-j} d(\alpha^j x)$ for $j=1,2.$}
    \label{fig: d and its rescaling}
\end{figure}

\subsection{Definition and summability of the lengths}\label{sec: Definition and summability of the lengths}

\begin{lemma}\label{lem: growth Ln}
For every integer $n$ with $|n|\geq3$, one has
\[
L_n\geq\sqrt{\frac{\alpha}{2}}\sqrt{|n|}.
\]
\end{lemma}

\begin{proof}
Let $j=j_n$ be the smallest nonnegative integer such that $|n|\alpha^{j+1}<1/2.$ Since $|n|\alpha>1/2$ whenever $|n|\geq3$, it follows that $j\geq1$. By the minimality of $j$, we have $|n|\alpha^j\geq1/2$, and hence $|n|\alpha^{j+1}\geq\alpha/2.$ Since $|n|\alpha^{j+1}<1/2$, we have $d\left(n\alpha^{j+1}\right)=|n|\alpha^{j+1}.$ 
Thus,
\[
L_n = \sum_m r^{-m} d(\alpha^m n) \ge r^{-(j+1)} d(\alpha^{j+1} n) = r^{(j+1)} n = \sqrt{\alpha n} \cdot \sqrt{\alpha^{j} n} \ge\sqrt{\frac{\alpha}{2}} \cdot \sqrt{n}.
\]

\end{proof}

\begin{remark}
    It is interesting to note that the rescalings 
    \begin{equation}\label{eq: rescaled F}
       r^k F(\alpha^{-k} x) = \sum_{m=-(k-1)}^{\infty} r^{-m} d(\alpha^m x)        
    \end{equation}
    converge (uniformly on any compact interval) to the limit 
    \[
        F_{\infty}(x) := \sum_{m=-\infty}^{\infty} r^{-m} d(\alpha^m x),
    \]
    that satisfies the self-similarity relation
    \[
        F_{\infty}(x) = r F_{\infty}(\alpha^{-1} x);
    \]
    see Figure~\ref{fig: F till 1000} for the graph of $F$ on $[0,1000]$.
    As $r=\alpha^{1/2}$, the occurrence of the power~$1/2$ in the conclusion of Lemma~\ref{lem: growth Ln} is quite natural.
\end{remark}

\begin{figure}
    \centering
    \includegraphics[width=0.7\linewidth]{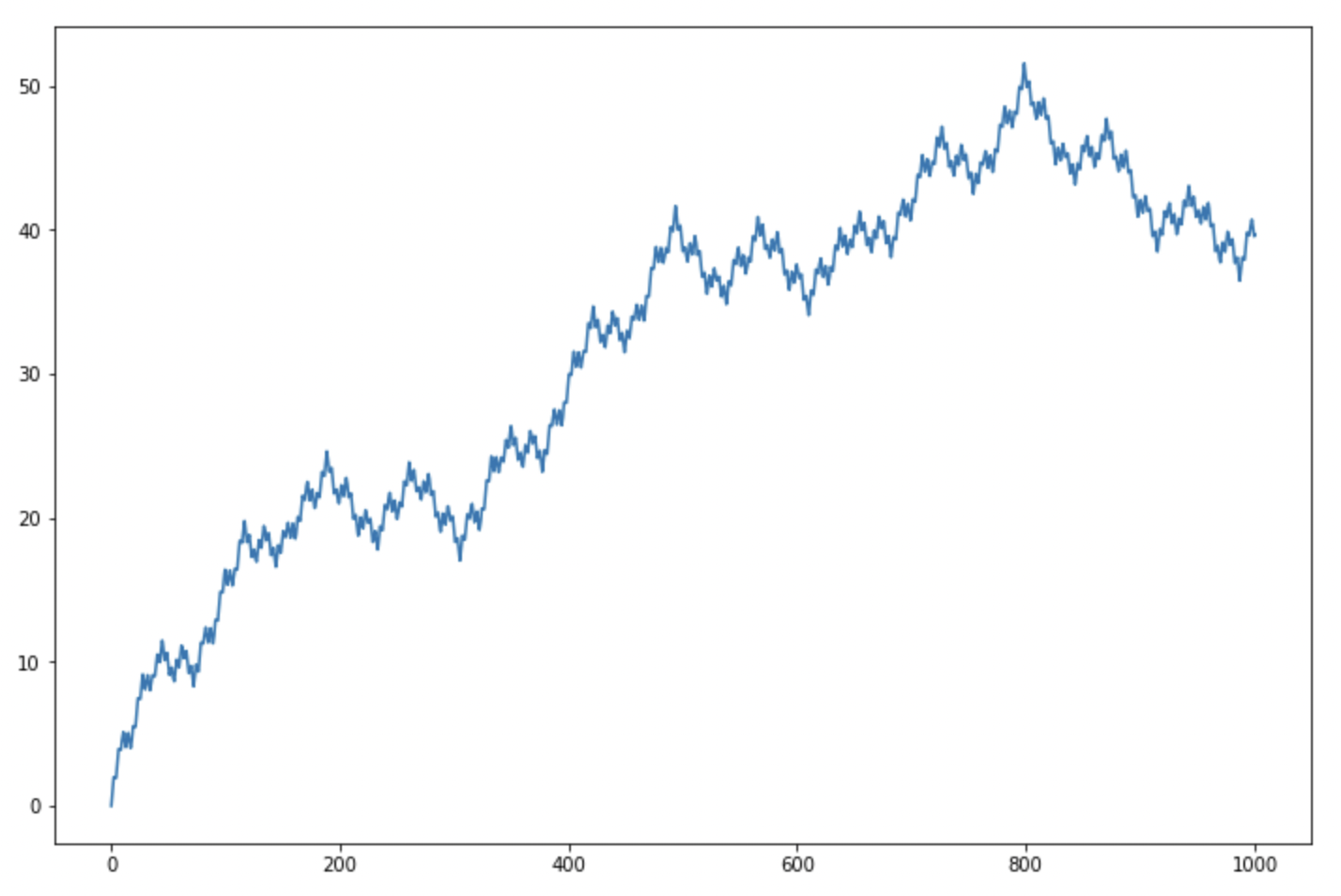}
    \caption{Graph of the function $F$ on $[0,1000]$.}
    \label{fig: F till 1000}
\end{figure}

\begin{proof}[Proof of Lemma~\ref{l: convergence and summability}]
The function $d(x)$ is 1-Lipschitz, and hence the summand $r^{-m} d(\alpha^m x)$ is $\alpha^m r^{-m} = r^m$-Lipschitz. Thus (as $d(0)=0$), it satisfies an upper bound by a geometric series
\[
r^{-m} d(\alpha^m x) \le r^m |x|,
\]
and thus the series~\eqref{eq: chosen L} converges. 

By Lemma~\ref{lem: growth Ln},
\[
\sum_{n\in\mathbb Z}\ell_n=\sum_{n\in\mathbb Z}e^{-L_n}\leq\sum_{n\in\mathbb Z} e^{-\sqrt{\frac{\alpha}{2}}\sqrt{|n|}}  <\infty.
\]
\end{proof}

\subsection{Gaps in the sequence of lengths}\label{sec: gaps}

This section is devoted to the proof of Proposition~\ref{prop: arithmetic gap}; once this is done, the proof of Theorem~\ref{thm: counterexample McDuff} will be complete. 

We start by making several remarks on the behaviour of the function~$F$. Let 
\begin{equation}\label{eq: define kappa}
    \kappa:=\sum_{j=1}^{\infty} r^j = \frac{r}{1-r}.
\end{equation}
Recall that $\alpha^{-1}=\sqrt{5}+2>4$, thus implying that $r=\sqrt{\alpha}<1/2$ and $\kappa=r/(1-r)<1$.
Now, we have the following statement.
\begin{lemma}\label{lem: initial line}
    The function $F$ is $\kappa$-Lipschitz. Moreover, $F(t)=\kappa t$ for all $t\in [0,1/(2\alpha)]$; in particular, $L_1=F(1)=\kappa$.
\end{lemma}
\begin{proof}
    The function $d(x)$ is $1$-Lipschitz, thus $r^{-j} d(\alpha^j x)$ is $r^j$-Lipschitz. The value of $\kappa$ in~\eqref{eq: define kappa} is thus the sum of the Lipschitz constants of the summands in the definition~\eqref{eq: define F} of~$F$. Also, all these summands are exactly linear on $[0,1/(2\alpha)]$.
\end{proof}

To describe the gap-creating behaviour of the function~$F$, let us introduce the following definitions.
\begin{definition}
We say that a point $c\ge 0$ is a \emph{separating point} for a function $G:[0,\infty)\to \R$, if $G(x)\le G(c)$ for every $x\in[0,c]$, and $G(x)\ge G(c)$
for every $x\ge c$.
\end{definition}
\begin{definition}
Say that the interval $[a,b]\subset [0,\infty)$ is \emph{$\lambda$-good} for the function~$G$, if every point $c\in [a,b]$ is a separating point for $G$, and additionally, the restriction $G|_{[a,b]}$ is affine with the slope $\lambda$: 
\[
G(c)=G(a)+\lambda (c-a)\quad \forall c\in [a,b],
\]
\end{definition}

Proposition~\ref{prop: arithmetic gap} will then follow from the following stronger statement.

\begin{prop}\label{prop: separating Mk intervals}
For every $k\geq 0$, the interval $\left[M_k-1/2,M_k+1\right]$ is $\kappa$-good for~$F$.

\end{prop}
We illustrate this proposition by the graph of the function $F$ on $[0,100]$: see Figure~\ref{fig: good intervals}, where $\kappa$-good intervals, corresponding to the first few values of $M_k$ (namely, to $1,5,22,94$), are shown.

\begin{figure}
    \centering
    \includegraphics[width=0.9\linewidth]{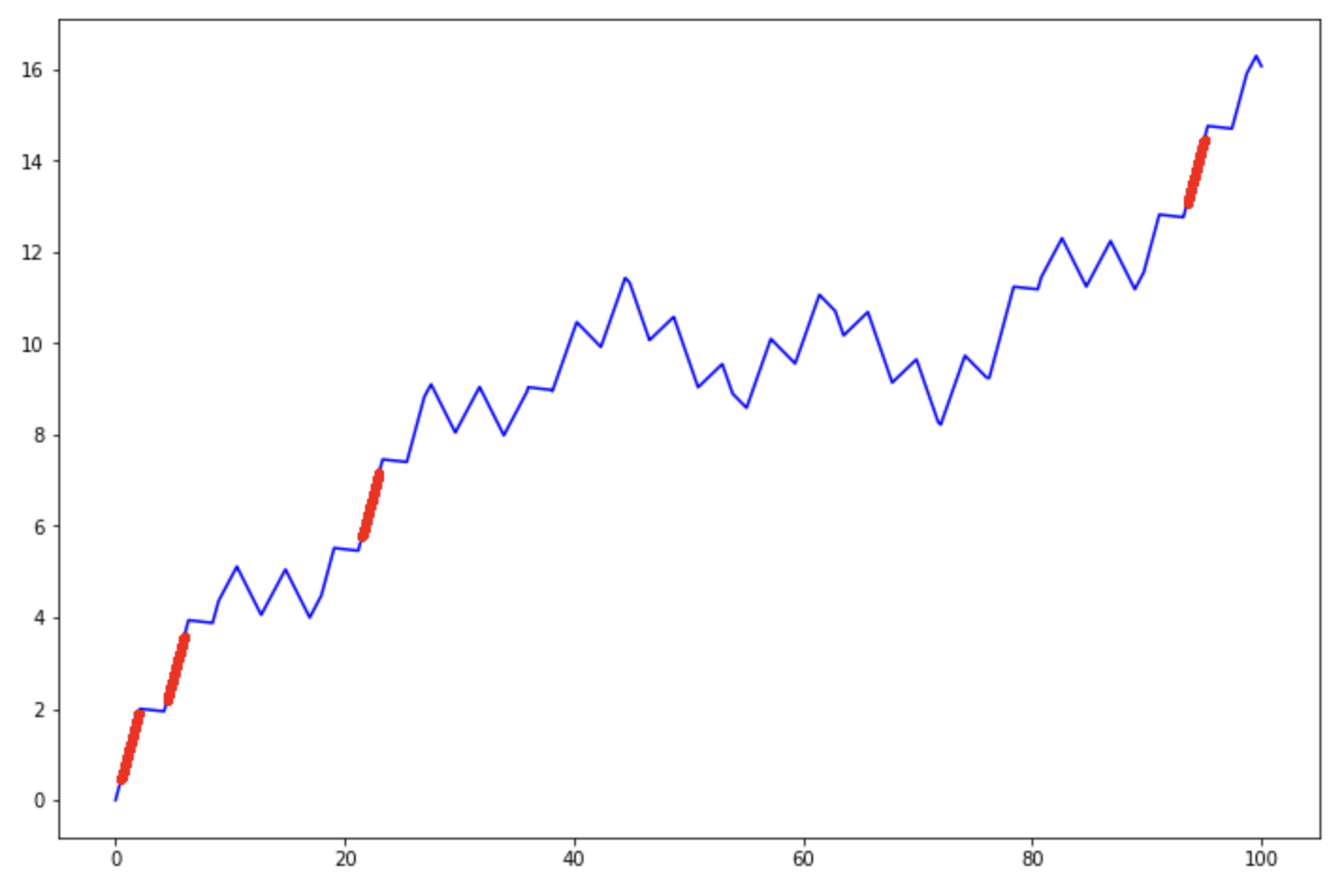}
    \caption{Graph of the function $F$ on $[0,100]$; bold red segments correspond to the intervals of the form $[M_k-1/2,M_k+1]$}
    \label{fig: good intervals}
\end{figure}

We will prove Proposition~\ref{prop: separating Mk intervals} by induction, using the self-similarity properties of the function~$F$. Namely, denote
\begin{equation}
    \tF(x):=d(x)+F(x)
    =\sum_{j=0}^{\infty} r^{-j} d(\alpha^j x);
\end{equation}
then, one has
\begin{equation}\label{eq: F,H}
F(x)=r^{-1}\tF(\alpha x).    
\end{equation}

The following lemma will be crucial for the induction step:

\begin{lemma}\label{lem: transfer separating interval}
Let $m\geq1$ be an integer, and suppose that the interval $\left[m-1/2,m+1\right]$ is $\kappa$-good for~$F$. Then the interval $[m+a,m+b]$ is $(1+\kappa)$-good for~$\tF$, where
\begin{equation}\label{eq: a,b}
a:=\frac{1-\kappa}{2(1+\kappa)}=\frac{1}{2}-r\quad\text{and}\quad b:=\frac{\kappa}{1+\kappa}=r.    
\end{equation}
\end{lemma}

\begin{proof}
As the function $F(x)$ is affine with the slope $\kappa$ on $[m-1/2,m+1]$, the sum $\tF(x)=F(x)+d(x)$ is affine with the slope $(\kappa+1)$ on $[m,m+1/2]$ (as $d$ is also affine on this interval), and in particular, on $[m+a,m+b]\subset[m,m+1/2]$; see Figure~\ref{fig: transfer}. The function $\tF$ is also affine with slope $\kappa-1<0$ on $[m-1/2,m]$ and on $[m+1/2,m+1]$.

Now, as $0\le d(x)\le 1/2$ for every~$x$, and as the interval $[m-1/2,m+1]$ is $\kappa$-good for $F$, we have for every $x\in [0, m-1/2]$
\[
    \tF(x)=F(x)+d(x) \le F\left(m-\frac{1}{2}\right)+\frac{1}{2} = \tF\left(m-\frac{1}{2}\right).
\]
Now, the value $a$ in~\eqref{eq: a,b} was defined so that $\tF(m+a)=\tF(m-1/2)$; again, see Figure~\ref{fig: transfer}. Indeed,
\[
\tF(m+a)-\tF(m) = a(1+\kappa) = \frac{1}{2} (1-\kappa) = \tF\left(m-\frac{1}{2}\right)-\tF(m).
\]
Hence, for every $x\in [0, m+a]$, one has $\tF(x)\le \tF(m+a)$.
\begin{figure}
    \centering
    \includegraphics[width=0.4\linewidth]{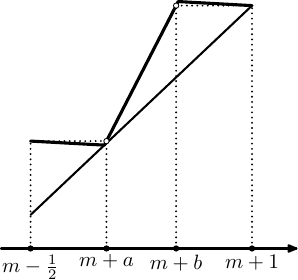}
    \caption{Graphs of $F$ (simple line) and $\tF$ (bold line) on the interval $[m-1/2,m+1]$ (see Lemma~\ref{lem: transfer separating interval}).}
    \label{fig: transfer}
\end{figure}

In the same way, for every $x\ge m+1$ one has 
\[
\tF(x) = F(x)+d(x) \ge F(x) \ge F(m+1) = \tF(m+1).
\]
The value $b<1/2$ was chosen in~\eqref{eq: a,b} so that 
\[
\tF(m+b)-\tF(m) = (1+\kappa) b = \kappa = \tF(m+1)-\tF(m),
\]
thus for every $x\ge m+b$ one has $\tF(x)\ge \tF(m+1)$.
\end{proof}

Using the rescaling relation~\eqref{eq: F,H} then immediately implies the following statement.
\begin{corollary}\label{cor: intervals}
    Let $m\geq1$ be an integer, and suppose that the interval $\left[m-1/2,m+1\right]$ is $\kappa$-good for~$F$. Then the interval $[\alpha^{-1}(m+a),\alpha^{-1}(m+b)]$ is $\kappa$-good for~$F$.
\end{corollary}
\begin{proof}
    The slope of the rescaled function $r^{-1}\tF(\alpha x)$ on the rescaled interval $[\alpha^{-1}(m+a),\alpha^{-1}(m+b)]$ differs from the slope of $\tF$ on $[(m+a),(m+b)]$ by the factor $r^{-1}\alpha=r$, and is thus equal to $r(1+\kappa)=\kappa$. Hence, this rescaled interval is $\kappa$-good for~$F$.
\end{proof}

We will use Corollary~\ref{cor: intervals} to prove Proposition~\ref{prop: separating Mk intervals} by induction. Namely, we will prove the following two statements.

\begin{lemma}[Base of induction]\label{lem: base}
    The interval $[1/2,2]$ is $\kappa$-good for~$F$.
\end{lemma}
\begin{lemma}[Step of induction]\label{lem: step}
    For every $k\ge 0$, one has 
    \begin{equation}\label{eq: Mk intervals inclusion}
        \left[M_{k+1}-1/2,M_{k+1}+1\right]
        \subset 
        \left[\alpha^{-1}(M_k+a),\alpha^{-1}(M_k+b)\right],
    \end{equation}
    where $a,b$ are given by~\eqref{eq: a,b}.
\end{lemma}

\begin{proof}[Proof of Lemma~\ref{lem: base}]
    By Lemma~\ref{lem: initial line}, the function $F$ coincides with $\kappa x$ on $[0,1/(2\alpha)]$. 
    Now,~\eqref{eq: F,H} implies that for every $x>0$,
    \begin{equation}\label{eq: F-r}
        F(x)= r^{-1} \tF(\alpha x) \ge r^{-1} F(\alpha x).
    \end{equation}
    For any $x> 1/(2\alpha)$ there exists a power $j>0$ such that $\alpha^j x \in \left[1/2, 1/(2\alpha)\right]$. Now, 
    \[
        \tF(x)\ge \tF(1) =\kappa \quad \text{ for every } \quad x \in \left[\frac{1}{2}, \frac{1}{2\alpha}\right]
    \]
    (this inequality is immediate on $[1,1/(2\alpha)]$, and follows from the negative slope of $\tF$ on $[1/2,1]$). Hence, from~\eqref{eq: F-r} we get 
    \[
        F(x) \ge r^{-j} \tF(\alpha^j x) \ge r^{-1} \cdot \min_{y \in \left[\frac{1}{2}, \frac{1}{2\alpha}\right]} \tF(y) = r^{-1} \cdot \kappa=\frac{1}{1-r}>1.
    \]
    In particular, the interval $[0,1]$ is $\kappa$-good for~$F$. 

    At this moment, we cannot formally refer to Lemma~\ref{lem: transfer separating interval}, starting from the interval $[0,1]$: the assumptions at the Lemma  at  left part of the interval are not satisfied, $[-1/2,1]$ is not a $\kappa$-good interval. However, we can repeat the arguments from its proof at the right part. Namely, the interval $[a,b]$ is $(1+\kappa)$-good for $\tF$, and thus the interval $[\alpha^{-1}a, \alpha^{-1}b]$ is $\kappa$-good for~$F$.
   
    Finally, it is straightforward to check that
    \begin{equation}\label{eq: inclusion}
        a<\frac{1}{2} \alpha <2\alpha< b.
    \end{equation}
    Indeed, note that $\alpha^{-1}<4+1/4$, thus implying $r^{-1}<2+1/16$ and $r>1/2-(1/16)(1/2^2)$. Hence, 
    \[
        a = \frac{1}{2}-r< \frac{1/16}{2^2}=\frac{1}{64}<\frac{\alpha}{2}.
    \]
    In the same way, 
    \[
        b=r>\frac{1}{2}-\frac{1}{64} > 2\cdot \frac{17}{72}= 2\cdot \frac{q_2}{q_3} >2\alpha.
    \]
    This proves~\eqref{eq: inclusion}. It implies the inclusion of the intervals 
    \[
    \left[\frac{1}{2},2\right]\subset 
    [\alpha^{-1} a, \alpha^{-1} b];
    \]
    hence, the interval $[1/2,2]$ is $\kappa$-good for~$F$.
\end{proof}

\begin{proof}[Proof of Lemma~\ref{lem: step}]
    Rescaling by $\alpha$, we see that the desired~\eqref{eq: Mk intervals inclusion} is equivalent to the inequalities
    \[
        M_k+a < \alpha\left(M_{k+1}-\frac{1}{2}\right), \qquad 
        \alpha(M_{k+1}+1) < M_k+b,
    \]
    or, equivalently,
    \begin{equation}\label{eq: alpha - M}
        a+\frac{\alpha}{2} < \alpha M_{k+1}-M_k < b-\alpha.
    \end{equation}
    Now, rewrite 
    \[
        \alpha M_{k+1} - M_k= \alpha \sum_{j=0}^{k+1} q_j - \sum_{j=0}^{k} q_j = \sum_{j=0}^{k+1} (\alpha q_{j}-q_{j-1}),
    \]
    where we have used $q_{-1}=0$ in the last equality.
    Now, recall that from Lemma~\ref{lem: exact approximation} we have $\alpha q_{j}-q_{j-1}=(-1)^{j}(\alpha)^{j+1}$. Thus,
    \begin{equation}\label{eq: sum - alpha}
        \sum_{j=0}^k (\alpha q_{j+1}-q_j) = \sum_{j=0}^k (-1)^j \alpha^{j+1} = \alpha-\alpha^2+\alpha^3-\dots +(-1)^{k+1} \alpha^{k+2}.        
    \end{equation}
    The sum in the right hand side of~\eqref{eq: sum - alpha} lies between $\alpha$ and $\alpha-\alpha^2$; straightforward check shows that 
    \[
        a+\frac{\alpha}{2} < \alpha-\alpha^2 <\alpha < b-\alpha.
    \]
    This implies~\eqref{eq: alpha - M} and thus the desired~\eqref{eq: Mk intervals inclusion}.
\end{proof}

\begin{proof}[Proof of Proposition~\ref{prop: separating Mk intervals}]
    The proof is by induction on~$k$. The base $k=0$ (for  $M_0=1$) is guaranteed by Lemma~\ref{lem: base}. Now, if for some $k\ge 0$ the interval $[M_k-1/2,M_k+1]$ is $\kappa$-good for $F$, then by Corollary~\ref{cor: intervals} the interval $[\alpha^{-1}(M_k+a),\alpha^{-1}(M_k+b)]$ is $\kappa$-good for~$F$. By Lemma~\ref{lem: step} this interval includes $[M_{k+1}-1/2, M_{k+1}+1]$; the latter is hence also $\kappa$-good for~$F$, thus completing the induction step.
\end{proof}

\section{Proof of Theorem~\ref{thm: counterexample McDuff}}

\begin{proof}
For the choice of lengths $\ell_n=e^{-L_n}=e^{-F(n)}$ as above, the new circle is well-defined, as the sum of their lengths converges due to Lemma~\ref{l: convergence and summability}. Due to Proposition~\ref{prop: continuous logarithmic slope - new}, the constructed map $f$ satisfies $\log Df=\varphi\circ p$, where $p$ is the projection on the initial circle (semiconjugating $f$ and $R_{\alpha}$), and hence $f$ is a $C^1$-diffeomorphism. Finally, the gaps in the sequence of log-lengths are present due to Proposition~\ref{prop: arithmetic gap}.
\end{proof}

\bigskip

\noindent{\bf Acknowledgments.} The author is very grateful to Leonardo Dinamarca and Sang-hyun Kim for many fruitful discussions. The author also wishes to thank Andrés Navas and Victor Kleptsyn for sending him a preliminary version of their closely related work, on which the present article builds. The author is especially thankful to Victor Kleptsyn for kindly explaining how to construct exceptional diffeomorphisms of the circle whose derivatives are not identically equal to one on their invariant Cantor sets, as well as for the many discussions that led to this work.

\end{document}